\documentclass[12 pt]{amsart}
\usepackage[pdftex,
colorlinks, bookmarksnumbered, bookmarksopen, linkcolor=blue, citecolor=blue]{hyperref}
\usepackage{amssymb}
\usepackage{amsmath,amsfonts,amsthm,amsopn,cite,mathrsfs}
\usepackage{epsfig,verbatim}
\usepackage{subfigure,color,enumitem}
\usepackage{pifont}
\newcommand{\fif}{if and only if}

\newcommand{\bdl}{band-limited}

\newtheorem{tm}{Theorem}[section]
\newtheorem{lemma}[tm]{Lemma}

\newtheorem{cor}[tm]{Corollary}
 \newtheorem{definition}[tm]{Definition}
 \newtheorem{rem}[tm]{Remark}

\newcommand{\beqa}{\begin{eqnarray*}}
\newcommand{\eeqa}{\end{eqnarray*}}

\DeclareMathOperator*{\supp}{supp}

\newcommand{\field}[1]{\mathbb{#1}}
\newcommand{\bR}{\field{R}}        %  real numbers
\newcommand{\bC}{\field{C}}        %  complex numbers
 \def\cF{\mathcal{F}}              % Calligraphic Letters

 \def\cH{\mathcal{H}}

 \def\cG{\mathcal{G}}

\def\rd{\bR^d}

\def\lrd{L^2(\rd)}

\def\intrd{\int_{\rd}}

\def\<{\left<}
\def\>{\right>}

\def\inv{^{-1}}

\def\mv1{M_v^1}

\newcommand{\ip}[2]{\ensuremath{\left<#1,#2\right>}}

\newcommand{\kl}{k_\lambda}
\newcommand{\mC}{\mathcal{C}}
\newcommand{\Gc}{\mathcal{G}}

\newcommand{\rkhs}{reproducing kernel Hilbert space}

\providecommand {\norm}[1] {\lVert#1\rVert}

\providecommand {\abs}[1] {\lvert#1\rvert}
\providecommand {\bigabs}[1] {\Bigl\lvert#1\Bigr\rvert}
\providecommand {\inprod}[1]{\langle #1 \rangle}
\providecommand {\biginprod}[1]{\Big\langle #1 \Big\rangle}
\providecommand {\set}[1]{\lbrace #1 \rbrace}

\providecommand {\inv}[1]{{#1}^{-1}}

\providecommand {\ind}{\mathbf{1}}

\begin{document}
\begin{abstract}
We derive necessary density conditions for sampling and for
interpolation in general \rkhs s satisfying some natural conditions on
the geometry of the space and the reproducing kernel. If the volume of
shells is small compared to the volume of balls (weak annular decay
property) and if the kernel possesses some off-diagonal decay or even  some
weaker form of  localization, then there exists a critical density $D$
with the following property: a set of sampling has density $\geq D$,
whereas a set of interpolation has density $\leq D$. The main  theorem
unifies many known density theorems in signal processing, complex
analysis, and harmonic analysis. For the special case of
bandlimited function we recover Landau's fundamental density
result. In complex analysis we rederive a critical density for
generalized Fock spaces. In harmonic analysis we obtain the first general
result about the density of coherent frames.
\end{abstract}

\title[Density in Reproducing Kernel Hilbert Spaces]{Density of Sampling and Interpolation in Reproducing Kernel
  Hilbert Spaces}
\author[H. F\"uhr]{Hartmut F\"uhr}
\author[K. Gr\"ochenig]{Karlheinz Gr\"ochenig}
\author[A. Haimi]{Antti Haimi}
\author[A. Klotz]{Andreas Klotz}
\author[J. L. Romero]{Jos\'e Luis Romero}

\address{Lehrstuhl A f\"ur Mathematik, RWTH Aachen, 52056 Aachen}
\email{fuehr@matha.rwth-aachen.de}

\address{Faculty of Mathematics \\
University of Vienna \\
Oskar-Morgenstern-Platz 1 \\
A-1090 Vienna, Austria}

\email{karlheinz.groechenig@univie.ac.at}
\email{antti.haimi@univie.ac.at}
\email{andreas.klotz@univie.ac.at}
\email{jose.luis.romero@univie.ac.at}

\address{Acoustics Research Institute, Austrian Academy of Sciences,
Wohllebengasse 12-14, A-1040, Vienna Austria}

\email{jlromero@kfs.oeaw.ac.at}
\subjclass[2010]{42C15,94A12,46C05,42C30,32A70}
\date{}
\keywords{Reproducing kernel Hilbert space, Beurling density, frame,
  sampling theorem, interpolation, square-integrable representation,
  Fock space, bandlimited function}
\thanks{K.\ G.\ and A.\ K.\ were
  supported in part by the  project P26273 - N25  of the
Austrian Science Fund (FWF).
J.\ L.\ R. gratefully acknowledges support from the Austrian Science Fund (FWF): P 29462 - N35
and from a Marie Curie fellowship, under grant PIIF-GA-2012-327063
(EU FP7). A. H. was supported by the Lise Meitner grant M1821 of the
Austrian Science Fund (FWF)
and by the WWTF grant INSIGHT (MA16-053)}

\maketitle

\section{Introduction}

How many samples of a function $f$ are necessary to completely recover
$f$? The first answer is the sampling theorem of Whittaker,
Kotelnikov, Shannon, and others~\cite{unser00}. It provides an explicit and
elegant reconstruction formula for the recovery of a bandlimited
function from its samples on a grid and establishes a fundamental
relation between the bandwidth of $f$ and the sampling density (the
Nyquist rate in engineering terminology). This sampling theorem is the
basis of modern information theory~\cite{unser00} and remains the model for
analog-digital and digital-analog conversion.

The decisive mathematical theorems work for more general notions of
bandwidth and for non-uniform sampling and are due to
Beurling~\cite{beurling66,seip04}
(sufficient conditions) and Landau~\cite{La}.   Landau's
necessary conditions give a precise meaning to the concept of a
Nyquist rate for bandlimited functions.

To this day, Landau's theorem is the prototype of a density theorem,
it has inspired  several hundred papers on
sampling. Landau's necessary conditions  have been transferred,
modified, and adapted to dozens of similar situations.
Here is a short, but by no means exhaustive list of density theorems in
the wake   of Landau:

\begin{enumerate}[itemindent=0cm, leftmargin=1.3cm, itemsep=0.2cm, parsep=0cm]
\item Sampling in spaces of analytic functions, in particular, in
  Bargmann-Fock space~\cite{seip92,seip-wallsten,lyub92} and in
  generalized Fock spaces~\cite{Abreu10,ortega1998beurling,Lin01}.
\item Sampling of bandlimited functions with derivatives~\cite{GR96,LOC14}.
\item Necessary density conditions of Gabor
  frames~\cite{daubechies90,RS95}. This topic alone has attracted about
  hundred  papers, for a detailed history of this density theorem
  we refer to Heil's survey article~\cite{heil07}.
\item Density conditions for abstract frames with some localization properties~\cite{bacahela06, BCHL06,bala07}.
\item Sampling in spaces of bandlimited functions on Lie groups~\cite{fugr07}.
\item Sampling in spaces of variable bandwidth~\cite{GK15}.
\item Sampling in spaces of bandlimited functions associated to an integral
  transform, e.g., the Hankel transform~\cite{AB12}.
\item Density of frames in the orbit of an irreducible  unitary representation
  of a homogeneous nilpotent Lie group~\cite{hoefler}.
\end{enumerate}
Essentially each of these contributions on necessary density
conditions for sampling and interpolation uses and  modifies
one of three methods.
\begin{enumerate}[itemindent=0cm, leftmargin=1.3cm, itemsep=0.2cm, parsep=0cm]
\item  Landau's original method is based on the spectral analysis of a
  family of localization operators (composition of the projection onto
  bandlimited functions with a time-limiting operator). This method is
  very powerful, but can
  become quite technical. Usually, the generalization of Landau's
  method is difficult.
\item The method of Ramanathan and Steger~\cite{RS95} was originally
  developed to prove the density theorem for Gabor frames. It compares
  and estimates   the dimension of finite-dimensional  subspaces
  corresponding to a local patch of a sampling set with the dimension
  of finite-dimensional  subspaces
  corresponding to a local patch of an interpolating set. This is  the
  method used most frequently. However, it is not universally
  applicable, because it requires the existence of a set that is
  simultaneously sampling and interpolating (or at least the
  construction of interpolating sets and sampling sets with almost the same density).
\item A third method goes back to Kolountzakis and
  Lagarias~\cite{KL96} (proof of Lemma 2.3)  who studied the density of tilings by
  translation. This method was then used by  Iosevich and
  Kolountzakis~\cite{IK06} to prove a version of Landau's theorem and   by
  Nitzan-Olevski \cite{NO12} to give the simplest proof of Landau's
  theorem. With hindsight this method consists of
  comparing  a set of sampling to a continuous frame. For us  this
 is the  method of choice that we will use in our investigation.
\end{enumerate}

We observe that all density theorems listed above treat certain Hilbert spaces
with a reproducing kernel (or isomorphic copies thereof).
This fact and the similarity
of all  proofs raises the question of a universal density theorem in
reproducing kernel Hilbert spaces.  This point of
view leads immediately to the pertinent  questions: what are the concrete
conditions on the underlying configuration space and on the reproducing kernel to
lead to a density theorem? What is the relevant density concept in a
\rkhs ?
Is there a  critical density in a \rkhs\ that separates sets of
sampling from sets of interpolation?

In this paper we attempt to give an answer to these questions and will
prove a general density theorem for functions in a \rkhs . Here is a
simplified version of our main result.

\begin{tm} \label{tmintro}
  Let $X$ be a metric measure  space   with a
  metric $d$ and a measure $\mu $ such that balls have finite measure,
 $\mu $ is non-degenerate, and satisfies the weak annular decay property, i.e., $\inf
  _x \mu
  (B_r(x)) >0$ for some $r>0$,  and $$ \lim_{r \to \infty} \sup_{x
    \in X} \frac{\mu(B_{r}(x) \setminus B_{r-1}(x))}{\mu(B_r(x))} = 0
  \, .
$$

Furthermore, let $\cH \subseteq L^2(X,\mu )$ be a \rkhs\  with a
reproducing kernel $k(x,y)$ satisfying
$\inf_{x\in X} k(x,x) >0$ and
an off-diagonal decay condition
of the form
\begin{equation}
  \label{eq:c11}
|k(x,y)| \leq C \big( 1 + d(x,y) \big)  ^{-\sigma} \qquad \text{ for all } x, y \in X \,
\end{equation}
for some $\sigma> 0$ satisfying $\lim_{r \to \infty} \sup _{x\in X}
\int _{X\setminus B_r(x) } (1+d(x,y))^{-2 \sigma} d\mu (y) =
0 $.

\begin{itemize}[itemindent=0cm, leftmargin=1cm, itemsep=0.2cm, parsep=0.15cm]
\item[(i)] Necessary conditions for sampling:  If for $\Lambda \subset X$  there exist $A,B >0$ such that
\begin{equation}
  \label{eq:e1}
A \|f\| ^2 \leq \sum _{\lambda \in \Lambda } |f(\lambda )|^2 \leq B
\|f\|^2 \qquad \text{ for all } f \in \cH   \, ,
\end{equation}
 then
$$
D^- (\Lambda ) := \liminf _{r\to \infty } \inf _{x\in X} \frac{\# (\Lambda
      \cap B_r(x))}{\mu (B_r(x))} \geq
\liminf _{r\to \infty } \inf_{x \in X}\frac{1}{\mu (B_r(x))} \int
  _{B_r(x)} k(y,y) d\mu (y) \, .
$$

\item[(ii)] Necessary conditions for interpolation:  Likewise, let
$\Lambda \subset X$ and assume that for every $a \in \ell^2(\Lambda)$, there exists a function $f \in \cH$
such that
\begin{equation*}
f(\lambda) = a_\lambda, \qquad \lambda \in \Lambda,
\end{equation*}
then
$$
D^+ (\Lambda ) := \limsup _{r\to \infty } \sup _{x\in X} \frac{\# (\Lambda
      \cap B_r(x))}{\mu (B_r(x))} \leq
\limsup _{r\to \infty } \sup _{x \in X}\frac{1}{\mu (B_r(x))} \int
  _{B_r(x)} k(y,y) d\mu (y) \, .
$$
\end{itemize}
\end {tm}

Following established terminology, we call a set $\Lambda \subseteq X$
that satisfies the sampling inequality~\eqref{eq:e1} a set of (stable)
sampling, while a set satisfying the interpolation property in part (ii)
of Theorem \ref{tmintro} is called a set of interpolation. Alternatively,  $\Lambda $ is   a set of sampling, \fif\
$\{k_\lambda :\lambda \in \Lambda \}$ is a frame for $\cH $, and $\Lambda
$ is a set of interpolation \fif\ $\{k_\lambda :\lambda \in \Lambda \}$ is a Riesz sequence in $\cH $,
i.e., if there exist $A,B >0$ such that
\begin{equation}
  \label{eq:e2}
A \|c\| ^2 \leq \| \sum _{\lambda \in \Lambda } c_\lambda \kl \|_2^2 \leq B
\|c\|^2 \qquad \text{ for all } c\in \ell ^2(\Lambda ).
\end{equation}

The densities $D^-(\Lambda )$ and $D^+(\Lambda )$ are the
obvious generalizations of the lower and upper Beurling density to
metric spaces.

The principal merit of Theorem~\ref{tmintro} is the clarification of
the main notions that go into a density theorem. To prove a density
theorem, one needs

\begin{itemize}[itemindent=1cm, leftmargin=0cm, itemsep=0cm, parsep=0.15cm]
\item[(i)] geometric data and the compatibility of metric and measure, and
\item[(ii)] estimates for the reproducing kernel.
\end{itemize}

The verification of these properties is by no means trivial. Indeed,
kernel estimates (Bergman, Bargmann, and other reproducing kernels)
constitute a deep and rich area of analysis. Theorem~\ref{tmintro}
shifts the emphasis in proofs of density theorems: it is important to
understand the geometry and the reproducing kernel, but it is no
longer necessary to prove a  ``new'' density theorem  from scratch with tedious modifications of known
techniques.

As an example we show how Landau's original theorem follows from Theorem~\ref{tmintro}. The discussion also shows
some of the difficulties in applying Theorem~\ref{tmintro}.

Let $X= \rd $ with Lebesgue measure $\mu $  and $\Omega \subseteq \rd $ be a
set of finite Lebesgue measure and define $B_\Omega = \{ f\in \lrd :
\supp \hat f \subseteq \Omega \}$ to be the subspace of \bdl\ functions
with spectrum in $\Omega $. Then
$$
f(x) = \int _\Omega \hat f (\xi ) e^{2\pi i x\cdot \xi  } \, d\xi =
\intrd f(y) \int _\Omega e^{2\pi i \xi (x-y)} \, d\xi dy \, ,
$$
and therefore $B_\Omega $ is a \rkhs\ with reproducing kernel $k(x,y)
= \int _\Omega e^{2\pi i \xi (x-y)} \, d\xi = \widehat{1_\Omega }(y-x)$. Clearly $\rd $ satisfies
all geometric conditions of Theorem~\ref{tmintro}. The computation of
the averaged trace is equally easy, since $k(x,x) = \int _\Omega 1 \,
d\xi  = \mu (\Omega )$ and thus $\frac{1}{\mu (B_r(x))} \int
  _{B_r(x)} k(y,y) d\mu (y) = \mu (\Omega )$ independent of $x$ and
  $r$. Thus a set of sampling must satisfy $D^-(\Lambda ) \geq \mu
  (\Omega )$ and a set of interpolation $D^+(\Lambda ) \leq \mu
  (\Omega )$, which is Landau's theorem.
The  decay condition \eqref{eq:c11} is satisfied for simple
spectra, e.g., when  $\Omega $ is  a cube or a convex set with smooth
boundary. Yet we must be cautious:  in general the kernel does not
  satisfy the  decay condition in~\eqref{eq:c11}, because
  the Fourier transform of the characteristic function of a  compact
  sets may decay
  arbitrarily slowly~\cite{Grah87}. In the main theorem (Theorem~\ref{main}) we  will impose  a
  much weaker condition on the kernel and thus will  take care of
  this subtle  point.

Most  density theorems of the  list above  can be
understood as an example of the general density theorem in \rkhs s. To
demonstrate the wide applicability of Theorem~\ref{tmintro} we
will rederive some of the fundamental density theorems in several
areas of analysis.

\begin{itemize}[itemindent=1cm, leftmargin=0cm, itemsep=0.1cm, parsep=0cm]
\item[(i)] \emph{Signal analysis}: as already indicated, Theorem~\ref{tmintro}
implies Landau's necessary density conditions for bandlimited
functions.

\item[(ii)] \emph{Complex analysis} (in several variables): we will deduce
Lindholm's density conditions~\cite{Lin01} for generalized Fock
spaces.

\item[(iii)] \emph{Harmonic analysis}: We will derive a necessary condition
for the density of a frame in the orbit of a square-integrable,
unitary representation of a group of polynomial growth. A special case
of this result is  the density theorem for Gabor frames.
\end{itemize}

Theorem~\ref{tmintro} is definitely not the end  of density
theorems. Of course, it is our main ambition to prove new density results.
In this sense the axiomatic set-up serves as a preparation for future
work. Currently there are numerous results about sampling theorems for
``sufficiently dense'' sets. See for instance,
\cite{Pe98,Pes01,Peza09, petrushev15,FFP16}. These results need to be complemented
by a critical density, provided that it exists at all.

Finally, let us point out some limitations.
We mention that the weak annular
decay property of the measure is not always satisfied, as it is tied
to the  growth of balls in $X$. Thus Theorem~\ref{tmintro}
excludes a number of very interesting examples, for instance, density
theorems in Bergman spaces~\cite{seip93,seip04,schuster00} and the
density of wavelet frames~\cite{Ku07}. However, in these cases the
Beurling  density is not  the correct density, and to this date it is an open
problem whether  a critical density  always  exists in this context. In
our view  the Beurling density is the correct notion of
density in geometries with polynomial (or subexponential)  growth, whereas in geometries
with exponential growth new phenomena arise (which are not at all
understood). See  Section~\ref{exa} for a hint.

The paper  is organized as follows:  In Section~2 we collect  the set
of assumptions for a general density theorem and then formulate the
main result. Section~3 contains a discussion of the main hypotheses
and several technical lemmas.   In
Section~4 we provide the proof of the main density theorem. In
Section~5 we apply the density theorem to rederive several fundamental
density theorems from the literature.

\section{The Density Theorem}

\subsection{Assumptions}
\label{sec_assumptions}

We first list the general assumptions on the geometry and the
reproducing kernel to make a density theorem work.

\smallskip

\noindent (A) {\bf Assumptions on the  metric and the  measure.}  We assume that
$(X,d)$ is a metric space, and $\mu$ is a measure on $X$
with the following properties:
\begin{itemize}[itemindent=0cm, leftmargin=1cm, itemsep=0.2cm, parsep=0.15cm]
\item   The metric $d:X\times X \to [0,+\infty)$ is a $\mu \otimes \mu$ measurable function and the balls $B_r(x) = \{ y \in X: d(y,x)< r \}$ satisfy $\mu(B_r(x))< \infty$ for all $r>0$ and $x \in X$.

\item \emph{Non-degeneracy of balls} (Axiom NDB): There exist $r>0$ such that
\begin{align}
\label{eq_NDB}
\inf_{x\in X} \mu(B_r(x)) >0\, .
\end{align}

\item \emph{Weak annular decay property} (Axiom WAD): Spherical shells have
  small volume compared to full balls:
  \begin{equation}
    \label{eq:e3}
 \lim_{r \to \infty} \sup_{x \in X} \frac{\mu(B_{r+1}(x) \setminus
   B_{r}(x))}{\mu(B_r(x))} = 0\, .
  \end{equation}
\end{itemize}

\smallskip

\noindent (B) {\bf Assumptions on  the reproducing kernel.} We assume that $\cH
\subseteq L^2(X,\mu )$ is a reproducing kernel Hilbert
space with reproducing kernel $k(x,y) $ so that
$$
f(x) = \int _X k(x,y) f(y) d\mu (y) = \langle f , k_x \rangle,
$$
where $k_x (y) = k(y,x)= \overline{k(x,y)}$. The assumptions on the
kernel are as follows:
\begin{itemize}[itemindent=0cm, leftmargin=1cm, itemsep=0.2cm, parsep=0.15cm]
\item Condition on the diagonal (Axiom D): there exist constants $C_1, C_2>0$ such that
for all $x \in X$
\begin{align}
\label{eq_AD}
C_1 \leq k(x,x) \leq C_2.
\end{align}

\item \emph{Weak localization of the kernel} (Axiom WL):
 For every $\epsilon >0$ there is a constant
 $r= r(\epsilon) $, such that
 \begin{equation}\label{eq:hapcont}
\sup _{x\in X } \int_{X \setminus B_r(x) }\abs{k(x,y)}^2 d\mu (y) <
\epsilon ^2.
\end{equation}

\item \emph{The homogeneous approximation property} (Axiom HAP):
   Assume that $\Lambda $ is such that $\{k_\lambda: \lambda \in \Lambda\}$ is
a Bessel sequence for $\cH$, i.e., $\sum _{\lambda \in \Lambda }
|f(\lambda ) | ^2 \leq C \|f \| ^2$ for all $f\in \cH $.  Then for every $\epsilon >0$ there
 is a constant $r= r(\epsilon )$, such that
 \begin{equation}\label{eq:hapclass}
\sup _{x\in X } \sum_{\lambda  \in \Lambda \setminus B_r(x)}\abs{k(x,\lambda )}^2  < \epsilon ^2 \,.
\end{equation}
\end{itemize}
We note that this version of axiom (HAP) differs from the usual
homogeneous approximation property used in the literature. Compare the
discussion in ~\cite{bacahela06}.

\subsection{Sets, densities and traces}
A set $\Lambda \subseteq X$ is called \emph{relatively separated} if there exists $\rho_0>0$ such that for all $\rho \geq \rho_0$, there exists $C=C_\rho>0$
such that:
\begin{equation}\label{eq:def_rel_sep}
\# (\Lambda \cap B_\rho(x)) \leq C \mu(B_\rho(x)),
\qquad x \in X.
\end{equation}

\begin{definition}
  The lower  Beurling density of a set $\Lambda \subseteq X $ is defined to
  be
  \begin{equation}
    \label{eq:2low}
    D^-(\Lambda ) = \liminf _{r\to \infty } \inf_{x\in X} \frac{\# (\Lambda
      \cap B_r(x))}{\mu (B_r(x))} \, ,
  \end{equation}
and the upper Beurling density of $\Lambda $ is
  \begin{equation}
    \label{eq:2up}
    D^+(\Lambda ) = \limsup _{r\to \infty } \sup_{x\in X} \frac{\# (\Lambda
      \cap B_r(x))}{\mu (B_r(x))} \, .
  \end{equation}
\end{definition}
Note that for relatively separated sequences, the upper Beurling density is finite by the
non-degeneracy of the balls \eqref{eq_NDB}.

We will compare  the density of a set of sampling or interpolation to
an invariant of the reproducing kernel Hilbert space.
The correct Nyquist rate is the averaged trace of the kernel. We
define the lower and upper traces as follows:
\begin{align}
  \label{eq:4}
  \mathrm{tr}^- (k) = \liminf _{r\to \infty } \inf_{x \in X}\frac{1}{\mu (B_r(x))} \int
  _{B_r(x)} k(y,y) d\mu (y) \, ,
  \\
  \label{eq:4b}
  \mathrm{tr}^+ (k) = \limsup _{r\to \infty } \sup_{x \in X}\frac{1}{\mu (B_r(x))} \int
  _{B_r(x)} k(y,y) d\mu (y) \, .
\end{align}

\subsection{Main result}
In this general context one can prove the following necessary density
conditions for sets of sampling and interpolation.

\begin{tm} \label{main}
Assume that $\cH $ is a reproducing kernel Hilbert space of functions
on a metric measure  space $X$
satisfying the  assumptions of Section \ref{sec_assumptions}.

\begin{itemize}[itemindent=0cm, leftmargin=1.3cm, itemsep=0.2cm, parsep=0cm]

\item[(i)] If $\Lambda \subseteq X$ is a set of stable sampling for $\cH $, then
\begin{equation}
\label{eq:5a}
D^-(\Lambda ) \geq    \mathrm{tr}^- (k) \mbox{ and }
D^+(\Lambda ) \geq    \mathrm{tr}^+ (k).
\end{equation}

\item[(ii)] If $\Lambda \subseteq X$ is a set of interpolation  for $\cH $, then
\begin{equation}
  \label{eq:5b}
D^-(\Lambda ) \leq   \mathrm{tr}^- (k) \mbox{ and }
D^+(\Lambda ) \leq   \mathrm{tr}^+ (k).
\end{equation}
\end{itemize}
\end{tm}

It is instructive to work out the trivial  case of a space  $X$ with  finite diameter, i.e., $X = B_R(x)$ for all $x
\in X$. Assuming the existence of a set $\Lambda$ of sampling in this case, we observe by Lemma \ref{lemma_sam_rel}
below that $\Lambda$ is finite, hence $\cH $ is finite-dimensional. Choosing any orthonormal basis
$\varphi_1,\ldots,\varphi_m$ of $\cH $, we have
 \[
  k(x,y) = \sum_{j=1}^m \varphi _j(x) \overline{\varphi _j(y)}~.
 \]

In particular, we get
\[
 \mathrm{tr}^- (k)  = \mathrm{tr}^+(k) = \frac{1}{\mu(X)} \int_X k(y,y) d\mu(y) = \frac{m}{\mu(X)}~.
\]

On the other hand, given any subset $\Lambda$, we have
\[
 D^+(\Lambda) = D^-(\Lambda) = \frac{\# \Lambda}{\mu(X)}~.
\]

Hence Theorem \ref{main} says that $\# \Lambda \le m$ is a necessary condition for sets of interpolation, whereas
$\# \Lambda \ge m$ is necessary for sets  of sampling (=uniqueness). Of course,
both
results follow from an elementary  dimension count: The  cardinality
of a Riesz sequence  is bounded by the dimension of the underlying  vector
space, whereas  a frame must  contain a basis and its cardinality exceeds the dimension of
$\cH $.

\section{Discussion of the Assumptions and Preliminary Lemmas}
In the following we always  assume the axioms from Section
\ref{sec_assumptions}.  We  discuss the axioms and prove some easy consequences.

\subsection{Metric and measure}
Most of the geometric conditions are technical conditions to exclude
pathologies.

First note that we \emph{do not} assume that the $\sigma$-algebra of $\mu$-measurable sets - the domain of $\mu$ - is generated by the class of open sets associated with the metric $d$. Rather, we only assume that the function $d$ is $\mu \otimes \mu$ measurable. This means that the sets $\{(x,y) \in X \times X: d(x,y) <r\}$ belong to the smallest $\sigma$-algebra generated by the sets $A \times B$, with $A,B$ in the domain of $\mu$. In particular, the balls
$B_r(x)$ are $\mu$ measurable. However, more general sets that are open with respect to the metric $d$ may not be $\mu$ measurable, since they may fail to be a countable union of balls.

The decisive condition is the weak annular decay property
(WAD). This condition links the metric and the measure and imposes
some compatibility between them. Even in simple examples, axiom (WAD)
requires some care. For example, the standard metric
$d_1(x,y) = |x-y|$ on $\mathbb{R}$ with Lebesgue measure fulfills
the weak annular decay property, but
the topologically equivalent  metric $d_2(x,y) = \log(1+|x-y|)$ violates (WAD).

\begin{rem} {\normalfont
For complicated geometries the
verification of the weak annular decay property is decidedly
non-trivial. In the literature one often uses the stronger
 {\em annular decay property} \cite{tes07}. A metric measure space
 $(X,d,\mu )$ satisfies this property, if there exist
 constants $C>0$ and $\delta\in(0,1]$ such that  for every $h \in
 [0,1]$, $r>0$, $x \in X$
 \begin{equation}
 \label{eq_ADP}
 \mu(B_{r}(x) \setminus B_{(1-h)r}(x))\leq C h^\delta\mu(\overline{B_r}(x)).
 \end{equation}
The annular decay property implies the weak annular decay property. Indeed,
if \eqref{eq_ADP} holds, the choice  $h = r^{-1}$ shows that
$\mu(B_{r}(x) \setminus B_{r-1}(x))\leq C r^{-\delta}  \mu(\overline{B_r}(x))$.
In addition, for $r  \geq 1$ and
$\varepsilon>0$,
\begin{align*}
\mu(\overline{ B_r}(x)) &\leq \mu(B_{r+\varepsilon}(x)) \leq \mu(B_{r}(x)) + C \varepsilon^{\delta}(r+\varepsilon)^{-\delta} \mu(\bar B_r(x))
\\ &\leq \mu(B_{r}(x)) +
C \varepsilon^{\delta} \mu(\overline{ B_r}(x)).
\end{align*}
So, choosing $\varepsilon $ small enough ($\varepsilon =
(2C)^{-1/\delta}$), we see that $\mu(\overline{ B_r}(x)) \leq C \mu(B_r(x))$.
Similarly, we see that $\mu(B_{r}(x)) \leq C' \mu(B_{r-1}(x))$, for $r \gg 1$
and a constant $C'>0$, so the conclusion follows.

In the literature one finds several conditions that imply the annular
decay property, for instance, if $X$ is a length space or if $X$ has
monotone geodesics. See~\cite{Buckley99,tes07} for further discussion
and more information on the annular decay property.
Axiom (WAD) seems to be tied to the growth of balls and seems  compatible
with at most subexponential growth.}
\end{rem}

We now note that the measure $\mu$ is \emph{locally doubling at large scales}.

\begin{lemma}
\label{lemma_loc_do}
There exist $r_0>0$ such that for all $r \geq r_0$, there is a
  constant $C_r>0$ such that
\begin{align}
\label{eq_ld}
\mu(B_{2 r}(x)) \leq C_r \mu(B_r(x))\, \quad \text{ for all } x\in X.
\end{align}
\end{lemma}
\begin{proof}
By Axiom (WAD), there exist $r_0$, such that for $r \geq r_0$
and all $x \in X$,
$\mu(B_{r+1}(x) \setminus B_r(x)) \leq \mu(B_r(x))$. As a consequence,
$\mu(B_{r+1}(x)) \leq 2 \mu(B_r(x))$. Iterating this estimate we conclude that
$\mu(B_{2r}(x)) \leq C_r \mu(B_r(x))$, where $C_r := 2^{\lceil r \rceil}$.
\end{proof}
\begin{rem}
\label{rem_ld}
{\normalfont Note that the proof of Lemma \ref{lemma_loc_do} only depends on Axiom (WAD).

The locally doubling property in \eqref{eq_ld} is much weaker than the usual doubling property for measures. See \cite{Tessera08, Coulhon95} for more on locally doubling spaces.
}
\end{rem}

We next formulate two lemmas on the number of points in and the
measure of general  ``spherical'' shells.

\begin{lemma} \label{annsep}
Let $\Lambda \subseteq X$ be relatively separated.
Then for all sufficiently large $\rho>0$, and
$R>\rho$, $r>0$, $x \in X$, we have
\label{lem:relsep}
  \begin{equation}
    \label{eq:relsep1}
      \# \Big(\Lambda \cap (B_{R+r}(x) \setminus B_{R}(x)) \Big)\leq
      C_{\rho,\Lambda} \, \mu (B_{R+r+\rho}(x) \setminus
B_{R-\rho}(x)) \,,
  \end{equation}
where the constant $C_{\rho,\Lambda}$  depends only on $\rho$ and
$\Lambda$.
  \end{lemma}
\begin{proof}
By Lemma \ref{lemma_loc_do}, for sufficiently large $\rho$, the locally doubling property
\begin{align}
\label{eq_a}
\mu(B_{\rho}(x)) \leq C_{\rho/2} \mu(B_{\rho/2}(x))
\end{align}
and the estimate in \eqref{eq:def_rel_sep} hold.

Let
$\set{B_{\rho /2} (y) \colon y \in X_0}$ be a maximal packing of $B_{R+r}(x) \setminus B_{R}(x)$.
This means that (i) $X_0 \subseteq B_{R+r}(x) \setminus B_{R}(x)$,
(ii) $\set{B_{\rho /2} (y) \colon y \in X_0}$ is a disjoint family of balls
and (iii) the family is maximal with respect to the properties (i) and (ii).
By maximality,
 $\set{B_{ \rho}(y) \colon y \in X_0}$ is a
covering of  $B_{R+r}(x)
\setminus B_{R}(x)$.
Using \eqref{eq:def_rel_sep} and \eqref{eq_a},
we obtain
\begin{align*}
\# \big(\Lambda \cap (B_{R+r}(x) \setminus B_{R}(x)) \big) &\leq
\# \big(\bigcup_{y \in X_0}\Lambda \cap B_{ \rho}(y) \big)
\leq C \sum_{y \in X_0} \mu( B_{ \rho}(y) ) \\
& \leq C C_{\rho/2} \sum_{y \in X_0}  \mu( B_{ \rho/2}(y) )= C C_{\rho/2}
\mu \left( \bigcup_{y \in X_0}  B_{ \rho/2}(y) \right) \\
&\leq   C  C_{\rho/2}  \, \mu (B_{R+r+\rho}(x) \setminus B_{R-\rho}(x)),
\end{align*}
where $C$ provided by \eqref{eq:def_rel_sep}  depends only  on $\rho$ and $\Lambda$.
\end{proof}

We will apply the weak annular decay property in the following
versions.
\begin{lemma} \label{lemma_annular_decay}
  If $(X,d,\mu )$ satisfies the weak annular decay property, then, for
  all $\rho ' >0$,
  \begin{equation}
   \label{eq:ann_decay}
   \lim_{r \to \infty} \sup_{x \in X} \frac{\mu(B_{r+\rho'}(x))}{\mu(B_r(x))} = 1~,
  \end{equation}
and
  \begin{equation}
    \label{eq:c1}
    \lim _{r\to \infty } \sup _{x\in X} \frac{\mu (B_{r+\rho '} (x)
      \setminus B_{r-\rho '} (x))}{\mu (B_r(x))} = 0 \, .
  \end{equation}
\end{lemma}
\begin{proof}
For the proof of (\ref{eq:ann_decay}), we observe that weak annular decay is equivalent to
\[
 \lim_{r \to \infty} \sup_{x \in X} \frac{\mu(B_{r+1}(x))}{\mu(B_r(x))} = 1~.
\]
For given $\rho'>0$, we set $N = \lceil \rho' \rceil$ and obtain
\begin{eqnarray*}
1 & \le & \sup_{x \in X}  \frac{\mu(B_{r+\rho'}(x))}{\mu(B_r(x))}  \le \sup_{x \in X}  \frac{\mu(B_{r+N}(x))}{\mu(B_r(x))} \\
& \le & \prod_{j=1}^N \underbrace{\sup_{x \in X} \frac{\mu(B_{r+j}(x))}{\mu(B_{r+j-1}(x))}}_{\to 1, \mbox{ as } r \to
\infty} ~.
\end{eqnarray*}
This proves (\ref{eq:ann_decay}). For the proof of (\ref{eq:c1}) note that
\begin{eqnarray*}
 0 & \le & \sup_{x \in X} \frac{\mu(B_{r+\rho'}(x)\setminus B_{r-\rho'}(x))}{\mu(B_r(x))} \\
 & \le & \sup_{x \in X}  \frac{\mu(B_{r+\rho'}(x))}{\mu(B_r(x))} - \inf_{x \in X} \frac{\mu(B_{r-\rho'}(x))}{\mu(B_r(x))}
\end{eqnarray*}
which converges to $1-1=0$ by  (\ref{eq:ann_decay}), as $r \to \infty$.
\end{proof}

\subsection{The reproducing kernel}

In practice, the upper bound in Axiom D \eqref{eq_AD} and the weak localization
property follow from off-diagonal decay estimates for the kernel (see
Section \ref{sec_apps}). To verify the lower bound in Axiom (D), the
following observation can be  useful.
\begin{lemma}
\label{obs_k_bb}
Let $k$ be a reproducing kernel on $X \times X$. If  there is a
constant $C>0$ such that for all $x \in X$, there exists
$f_x \in X$ such that $\norm{f_x} \leq C$ and $f_x(x)=1$, then the lower bound in Axiom (D) holds.
\end{lemma}
\begin{proof}
The claim  follows from  $1=f_x(x)=\ip{f_x}{k_x} \leq C \norm{k_x} = C
k(x,x)^{1/2}$.
\end{proof}

\noindent \emph{Normalization of the reproducing kernel.}
In some  examples, e.g., in Fock spaces of entire functions, the
reproducing kernel is unbounded. This situation can be dealt with by
the following normalization. Let $\psi : X \to \bR ^+$ and define a
new measure $\tilde \mu = \psi ^2 \mu $. Then $J$ defined by $Jf = f
\psi \inv $ is an isometry from $L^2(X,\mu )$ onto $L^2(X,\tilde \mu
)$ and  $\tilde \cH = J\cH\subseteq L^2(X,\tilde \mu )$ is again a
\rkhs . We calculate the new  reproducing kernel $\tilde
k$ as follows:
$$
Jf(x) = \psi (x) \inv f(x) = \psi (x) \inv \langle f, k_x\rangle =
\psi (x) \inv \langle Jf, Jk_x\rangle \, ,
$$
whence the new reproducing kernel is
\begin{equation}
  \label{eq:v3}
\tilde k(x,y) = \psi (x) \inv \psi (y)\inv k(x,y) \qquad   x,y \in X
\, .
\end{equation}
In particular, if we choose $\psi (x) = \|k_x\|$, then
$\tilde k(x,y) = \|k_x\| ^{-1} \|k_y\|\inv  k(x,y)$ and thus
$ \tilde k(x,x) = 1$.  Axiom D \eqref{eq_AD} can therefore always be fulfilled by
a renormalization of the kernel. However, this may be at the price of
destroying some other required properties. Note that in this
normalization the critical density is always one, provided that
Theorem~\ref{main} is still applicable. See also \cite{chha13}
for normalization of weighted Bergman spaces on the disk.

\begin{lemma}
\label{lemma_sam_rel}
Let
$\Lambda \subseteq X$ be a set such that $\{k_\lambda: \lambda \in
\Lambda\}$ is a Bessel sequence in $\cH $. Then
$\Lambda$ is relatively separated.
\end{lemma}
\begin{proof}
Let $\epsilon^2 := \tfrac{1}{2}\inf_{x\in X} k(x,x)$. By Axiom D \eqref{eq_AD}, $\epsilon>0$.
Select $r_1=r_1(\epsilon)$ according
to Axiom (WL). Since $\int _X |k(x,y)|^2 d\mu (y) = \|k_x\|^2 = k(x,x)$,
we obtain that, for $\rho \geq r_1$,
\begin{align}\label{eq:v5}
\int_{B_\rho(x)}\abs{k(x,y)}^2 d\mu (y) = k(x,x) -  \int_{X \setminus B_\rho (x)}\abs{k(x,y)}^2 d\mu (y) \geq \epsilon^2.
\end{align}
By hypothesis, there exists a constant $C>0$ such that for all $f \in
\cH$
\begin{align*}
\sum_{\lambda \in \Lambda} \abs{f(\lambda)}^2 \leq C \norm{f}^2 \, ,
\end{align*}
and this holds in particular for $f = k_y$.
Using Lemma \ref{lemma_loc_do}, select $r_0$ such that \eqref{eq_ld} holds for $r \geq r_0$. Let $x \in X$ and $\rho
\geq \rho_0 := \max\{r_0,r_1\}$. Using \eqref{eq:v5} and Axiom D \eqref{eq_AD}, we
estimate
\begin{align*}
\epsilon ^2  \# (\Lambda \cap B_{\rho}(x))
&\leq \sum_{\lambda \in \Lambda \cap B_{\rho}(x)} \int_{B_{\rho}(\lambda)}\abs{k(\lambda,y)}^2 d\mu (y)
\\
&\leq \sum_{\lambda \in \Lambda \cap B_{\rho}(x)} \int_{B_{2\rho}(x)}\abs{k(\lambda,y)}^2 d\mu (y)
\leq \int_{B_{2\rho}(x)}
\sum_{\lambda \in \Lambda} \abs{k_y(\lambda)}^2 d\mu (y)
\\
&\leq C  \,  \sup_{y\in X} \norm{k_y}^2 \mu(B_{2\rho}(x)) \leq C C_{\rho}
\mu(B_{\rho}(x)).
\end{align*}
Hence, \eqref{eq:def_rel_sep} holds.
\end{proof}
\begin{rem}
\label{rem_lemma_nohap}
{\normalfont
Note that the proof of Lemma \ref{lemma_sam_rel} does not use Axiom (HAP).
}
\end{rem}

We will  need the following elementary facts about  frames and Riesz
sequences  in a reproducing kernel Hilbert space $\cH$. They are
copied from ~\cite{NO12,GK15}.

\begin{lemma}\label{lem:RB-kernel}
The following properties hold.

\begin{itemize}[itemindent=0cm, leftmargin=1.3cm, itemsep=0.2cm, parsep=0cm]

\item[(i)] Assume that $\{ k_\lambda: \lambda\in \Lambda \}$ is a frame for $\cH$ with
canonical dual frame $\{g_\lambda: \lambda\in \Lambda \}$.
Then $k_\lambda $ and $g_\lambda$ satisfy the following:
\begin{align}
 & \sum_{\lambda \in \Lambda} k_\lambda(y)  \overline{g_\lambda(y)} =  k(y,y),
\qquad y \in X, \label{eq:c41} \\
&  \sup _{y\in X} \sum _{\lambda \in \Lambda } |g_\lambda (y)|^2 <\infty,
\qquad y \in X,\label{eq:c42}  \\
& \sup _{\lambda\in \Lambda} \norm{g_\lambda} = C <\infty, \label{eq:c43} \\
& \sup_{\lambda \in \Lambda}\abs{\langle k_\lambda, g_\lambda \rangle } \leq 1.  \label{eq:c44}
\end{align}

\item[(ii)]  If $\{ k_\lambda: \lambda\in \Lambda\}$ is a Riesz basis for a subspace
$V\subseteq \cH$ with biorthogonal basis  $\{g_\lambda: \lambda\in \Lambda\}\subseteq V$,
then \eqref{eq:c42}, \eqref{eq:c43} hold true, while
\eqref{eq:c41} is replaced by the inequality
\begin{equation}
  \label{eq:c45}
    0 \leq \sum_{\lambda \in \Lambda} k_\lambda(y) \overline{g_\lambda(y)} \leq   k(y,y),
\end{equation}
and in \eqref{eq:c44} holds the equality
\begin{align} \label{eq:c44a}
\langle g_\lambda, k_\lambda \rangle =1  \mbox{ for all }\lambda\in \Lambda.
\end{align}
\end{itemize}
\end{lemma}

\begin{proof}
The proof from \cite{NO12,GK15} is included for completeness. Let
$P_V$ be the orthogonal projection on the subspace $V$ of $\cH $.
Inequality \eqref{eq:c45} follows from
 \begin{align*}
    \sum_{\lambda \in \Lambda} k_\lambda(y) \overline{g_\lambda(y)}
=& \sum_{\lambda \in \Lambda}\inprod{k_\lambda, k_y}\inprod{k_y, g_\lambda}\\
=& \biginprod{\sum_{\lambda \in \Lambda} \inprod{k_y, g_\lambda} k_\lambda, k_y}
= \inprod{P_V k_y,k_y} \\
&\leq \| k_y\|^2 = k(y,y) \,.
  \end{align*}
The proof of \eqref{eq:c41} is the same, except that $P_V =
\mathrm{I}$ for frames and thus equality holds in the last step.

Item \eqref{eq:c42} follows from
\[
\sum_{\lambda \in \Lambda}\abs{g_\lambda(y)}^2=\sum_{\lambda \in
  \Lambda}\abs{\inprod{g_\lambda, k_y}}^2\leq C\norm{k_y}^2 = C k(y,y)\,,
\]
where $C$ is the upper frame bound for $\set{g_\lambda \colon \lambda \in \Lambda}$,
and $k(y,y)$ is uniformly bounded by Axiom D \eqref{eq_AD}.

Item \eqref{eq:c44}
 is an immediate consequence of the minimality of the $\ell^2$-norms of the
 coefficients in the canonical  frame expansion~\cite{DS52}:
\[
k_{\lambda'}=\sum_{\lambda \in \Lambda}\inprod{k_{\lambda'},g_\lambda} k_\lambda  =
1 \cdot   k_{\lambda'} \quad \text{ for every } \lambda' \in \Lambda\,,
\]
so
\[
\abs{\inprod{k_{\lambda'},g_{\lambda'}}}^2 \leq \sum_{\lambda \in
  \Lambda}\abs{\inprod{k_{\lambda'},g_{\lambda}}}^2 \leq 1\quad \text{ for every
} \lambda' \in \Lambda\,.
\]
Finally \eqref{eq:c43} is a general fact about frames.
\end{proof}

\section{Proof of Theorem \ref{main}}
In this section we prove the necessary density conditions for sets of
sampling or interpolation in general \rkhs s satisfying the conditions
of  Section~\ref{sec_assumptions}.  Similar to~\cite{IK06,NO12}, our
proof is inspired by the method of Kolountzakis and
Lagarias~\cite{KL96}. It  is modeled on our
own version in~\cite{GK15}.

\begin{proof}
\textbf{Step 1.} The first part of the proof works both for sets of
sampling and for sets of interpolation. Equivalently, we assume that $\{k_{\lambda }
:\lambda \in \Lambda \}$ is either a frame with \emph{canonical dual
  frame} $\{ g_\lambda :\lambda \in \Lambda \}$ or that $\{k_{\lambda }
:\lambda \in \Lambda \}$ is a Riesz sequence for some subspace $V
\subseteq \cH $  with \emph{biorthogonal basis } $\{ g_\lambda :\lambda \in \Lambda
\}\subseteq V$. By Lemma \ref{lemma_sam_rel}, the set $\Lambda$ is
relatively separated. Let $\rho>0$ be a suitably large radius, such that the conclusion of Lemma \ref{annsep} holds.

In both cases we estimate the quantity
\begin{equation}
  \label{eq:6}
  \int_{B_r(x)} \sum_{\lambda  \in \Lambda } k_{\lambda}(y)
    \overline{g_\lambda (y)}   d\mu (y)  \,
\end{equation}
for large $r$.

Fix $\epsilon >0$ and $x\in X$,  and choose $R=R(\epsilon )$ such that both kernel axioms
WL~\eqref{eq:hapcont} and HAP~\eqref{eq:hapclass} are satisfied.
In the proof, we will just write $B_r$ for the ball $B_r(x)$ to abbreviate the notation.

We partition $\Lambda $ and write accordingly
\begin{align*}
\sum_{\lambda \in \Lambda } k_\lambda(y)
\overline{g_\lambda(y)} & =\Big(\sum_{\lambda \in \Lambda \cap B_{r-R}
}+\sum_{\lambda \in \Lambda \cap (X \setminus  B_{r+R} )}+\sum_{\lambda
  \in \Lambda  \cap(B_{r+R} \setminus B_{r-R})} \Big) k_\lambda(y)
\overline{g_\lambda (y)}\\
 & = A_1(y)+A_2(y)+ A_3(y) \, .
\end{align*}

\emph{Estimate of $\int _{B_r} A_1$.}
We estimate $\abs{\int_{B_r} A_1(y) d\mu (y)}$. We write
$$
\int_{B_r} A_1(y) d\mu (y) = \int _X \sum_{\lambda \in \Lambda \cap B_{r-R} } k_\lambda(y)
\overline{g_\lambda(y)}  d\mu (y) - \int _{X\setminus B_r} \sum_{\lambda \in \Lambda \cap
B_{r-R}
}  k_\lambda(y)
\overline{g_\lambda(y)}  d\mu (y) \, ,
$$
 and set
$$ L= \sum_{\lambda \in \Lambda \cap B_{r-R}
}  \int _{X\setminus B_r}  k_\lambda(y) \overline{g_\lambda(y)}  d\mu (y) \, .
$$
Then
\begin{equation}
  \label{eq:c10}
\int_{B_r} A_1(y) d\mu (y)  = \sum_{\lambda \in \Lambda \cap B_{r-R}
}  \langle k_\lambda, g_\lambda \rangle  - L \, .
\end{equation}
If  $\lambda \in \Lambda \cap B_{r-R}$ and $y \in X \setminus B_r$, then
$d(\lambda , y) > R$. Therefore the kernel axiom
WL~\eqref{eq:hapclass} and (\ref{eq:c43}) imply that  a single term
contributing to $L$ is majorized by
\begin{align} \label{eq:a0}
\bigabs{\int_{X \setminus B_r }  k_\lambda(y) \overline{g_\lambda(y) }d\mu (y)}  \leq
\Big( \int _{X \setminus B_R (\lambda) } |k_\lambda(y)|^2 \, d\mu (y) \Big)^{1/2}
\|g_\lambda\| \leq \epsilon C'' \, .
  \end{align}
This estimate  implies
\begin{align} \label{eq:a1}
|L| \leq \epsilon \, C_1 \, \# (\Lambda \cap B_{r-R}) \leq \epsilon \, C_1 \,
\# (\Lambda \cap B_r) \, .
\end{align}

\emph{Estimate of $\int_{B_r} A_2$.}  Note that $y\in B_r$ and $\lambda \in
\Lambda
\setminus B_{r+R}$ implies that $d(\lambda ,y)
>R$. Then  Axiom~HAP~\eqref{eq:hapclass} ensures
that $\sum _{\lambda \in \Lambda \setminus B_{r+R}} |k(y, \lambda)|^2
\leq \sum _{\lambda \in \Lambda \setminus B_R(y)} |k( y, \lambda )|^2 < \epsilon^2$. Consequently, using also
\eqref{eq:c42}, we obtain
\begin{align}  \label{eq:a2}
  \bigabs{\int_{B_r} A_2(y)d\mu( y) } &\leq \int_{B_r} \Big(\sum_{\lambda \in
    \Lambda  \cap(
    X \setminus B_{r+R})}\abs{ k_\lambda(y) }^2\Big)^{1/2}\Big(
  \sum_{\lambda  \in     \Lambda }\abs{g_\lambda (y)}^2\Big)^{1/2} d\mu (y)
 \leq \epsilon C_2 \mu (B_r)  \, .
\end{align}

\emph{Estimate of $\int_{B_r} A_3$.}
For the third term observe that
\begin{align}
\int_{B_r} \abs{A_3(y)} d \mu (y)
&\leq \sum_{\lambda \in \Lambda \cap( B_{r+R}\setminus B_{r-R})}
\int_X \abs{ k_\lambda(y) }\, \abs{g_\lambda(y)}d\mu (y) \notag  \\
&\leq  \sum_{\lambda \in \Lambda \cap( B_{r+R}\setminus B_{r-R})} \norm { k_\lambda }
\norm{g_\lambda}  \, .
\end{align}
Using Axiom D \eqref{eq_AD}, and the boundedness of
the canonical dual frame~\eqref{eq:c43}, we obtain
\begin{equation}
  \label{eq:a3}
\int_{B_r} \abs{A_3(y)} \, d\mu (y) \leq  C_3 \# (\Lambda \cap (B_{r+R} \setminus
B_{r-R})).
  \end{equation}

From now on we distinguish the case of sets of sampling from sets of
interpolation.

\textbf{Step 2.} Assume first that $\{ k_\lambda : \lambda \in
\Lambda \}$ is a Riesz sequence in $\cH $.
We rewrite  the expansion
\begin{align*}  \int _{B_r} \sum _{\lambda \in \Lambda } k_\lambda(y)
  \overline{g_\lambda (y)} \, d\mu (y)  = \int_{B_r} \sum _{j=1}^3 A_j(y)
  \, d\mu (y) \, ,
\end{align*}
and, with the help of \eqref{eq:c10}, \eqref{eq:c45}, and (\ref{eq:c44a}), we   obtain
the estimate
\begin{align*}
  \# (\Lambda \cap B_{r-R}) &= \sum _{\lambda \in \Lambda \cap B_{r-R}}
  \langle k_\lambda, g_\lambda \rangle \\
&= \int _{B_r} A_1(y) \, d\mu (y) + L \\
&= \int _{B_r} \sum _{\lambda \in \Lambda } k_\lambda(y)
  \overline{g_\lambda (y)} \, d\mu (y) - \int _{B_r} A_2(y) d\mu (y) -
   \int _{B_r} A_3(y) d\mu (y) + L \\
&\leq \int_{B_r} k(y,y) \, d\mu (y)  + \Big| \int _{B_r} A_2(y) d\mu (y)
\Big| + \Big|  \int _{B_r} A_3(y) d\mu (y) \Big| + |L| \, .
\end{align*}

Using $\# (\Lambda \cap B_r) = \# (\Lambda \cap B_{r-R}) \, + \,  \#
(\Lambda \cap
(B_r \setminus B_{r-R})) $ and the estimates for  $\int _{B_r} A_j (y) d\mu (y)$
(see \eqref{eq:a2}, \eqref{eq:a3} and \eqref{eq:a1}), we obtain that
\begin{multline}
\# (\Lambda \cap B_r) \leq \int _{B_r} k(y,y) \, d\mu (y) +\epsilon C_2 \mu
(B_r) + C_3 \, \# (\Lambda \cap  (B_{r+R} \setminus B_{r-R})) \\ + \epsilon C_1 \# (\Lambda \cap B_r)
+ \# (\Lambda \cap  (B_r \setminus B_{r-R})).
\end{multline}
Lemma~\ref{lem:relsep} bounds the last term by
$$
\# (\Lambda \cap  (B_r \setminus B_{r-R}))  \leq \# (\Lambda \cap  (B_{r+R}
\setminus B_{r-R})) \leq C_{\rho,\Lambda} \,  \mu  (B_{r+R+\rho}\setminus
B_{r-R-\rho}) \, ,
$$
so we conclude that
\begin{equation}\label{eq:8}
(1-\epsilon C_1)
\frac{\# (\Lambda \cap B_r)}{\mu (B_r)} \leq \frac{1}{\mu (B_r)} \int _{B_r}
k(y,y) \, d\mu (y) +\epsilon C_2 + (1+C_3) C_{\rho,\Lambda} \frac{ \mu
  (B_{r+R+\rho}\setminus  B_{r-R-\rho})}{\mu(B_{r})}.
\end{equation}

We recall that $B_r = B_r(x)$,  take the supremum over all
$x\in X$, let $r$
tend to $\infty$,  and use Lemma \ref{lemma_annular_decay} to deduce
\begin{align*}
  (1-\epsilon C_1) D^+ (\Lambda ) & \leq \mathrm{tr}^+(k) + \epsilon C_2 + (1+C_3)C_{\rho,\Lambda} \limsup_{r \rightarrow \infty} \sup_{x\in X}
\frac{\mu (B_{r+R+\rho}(x)\setminus  B_{r-R-\rho}(x))}{\mu(B_{r}(x))}
\\
&= \mathrm{tr}^+(k) + \epsilon C _2.
\end{align*}

Since $\epsilon>0$ is arbitrary, it follows that $D^+(\Lambda ) \leq
\mathrm{tr}^+(k)$ for every interpolating set $\Lambda $. The inequality
$D^-(\Lambda ) \leq \mathrm{tr}^-(k)$ follows from \eqref{eq:8} in a similar way,
just taking $\inf$ instead of $\sup$ and $\liminf$ instead of $\limsup$.

\textbf{Step 3.} Assume next  that $\{ k_\lambda : \lambda \in
\Lambda \}$ is a frame  for  $\cH $. Then by
Lemma~\ref{lem:RB-kernel}, \eqref{eq:c41} and \eqref{eq:c44},  we
have
\[
\sum _{\lambda \in \Lambda } k_\lambda(y) \overline{g_\lambda (y)}  = k(y,y)
\]
and
\[
\abs{\langle k_\lambda, g_\lambda \rangle} \leq  1 \, .
\]

Proceeding as in Step~2 we obtain with $x \in X$ fixed and $B_r =
B_r(x)$ that
\begin{align*}
&\int _{B_r} k(y,y) d\mu (y) = \int _{B_r} \sum _{\lambda \in \Lambda }
k_\lambda(y) \overline{g_\lambda (y)} \,d\mu (y)
= \int _{B_r}  \sum _{j=1}^3 A_j(y)
  \, d\mu (y) \\\
&\qquad = \sum_{\lambda \in \Lambda \cap B_{r-R}}  \langle k_\lambda , g_\lambda \rangle  - L  + \int _{B_r}
A_2(y) \, d\mu (y) + \int _{B_r} A_3(y) \, d\mu (y) \\
 &\qquad\leq  \# (\Lambda \cap B_{r-R}) + \epsilon C_1 \, \# (\Lambda \cap B_r) +
\epsilon C_2 \mu (B_r) + C_3 \, \# (\Lambda \cap (B_{r+R} \setminus B_{r-R}))
 \\
&\qquad \leq  (1+\epsilon C_1) \# (\Lambda \cap B_r)+
\epsilon C_2 \mu (B_r) + C_3 C_{\rho ,\Lambda } \, \mu
(B_{r+R+\rho} \setminus B_{r-R-\rho}) \, .
\end{align*}
Consequently
\begin{align} \label{eq:e5}
\frac{1}{\mu(B_r)}\int _{B_r} k(y,y) d\mu (y) \leq
(1+ \epsilon C_1 )  \frac{\# (\Lambda \cap B_r)}{\mu (B_r)} +
\epsilon C_2 + C_3   C_{\rho ,\Lambda } \,\frac{ \mu
(B_{r+R+\rho} \setminus B_{r-R-\rho}) }{\mu
  (B_r)} \, .
\end{align}

Again, we  take the infimum over all $x\in X$ and let $r$
tend to $\infty$ to obtain via Lemmas \ref{lem:relsep} and \ref{lemma_annular_decay}
$$
\mathrm{tr}^- (k) \leq (1 + \epsilon C_1 ) D^-(\Lambda ) + \epsilon C_2 \, .
$$
Since $\epsilon >0$ was arbitrary, the necessary density is
$D^-(\Lambda ) \geq \mathrm{tr}^- (k)$, as claimed. As in Step 2, the statement
involving the upper trace and density follows by
just taking $\sup$ instead of $\inf$ and $\limsup$ instead of $\liminf$.
\end{proof}

By drawing a different  conclusion at the end of the above proof, the density
theorem can be given a dimension-free form as  suggested to us by
J.~Ortega-Cerd\`a.

\begin{cor}
  Impose the same assumption on $(X,d,\mu)$ and $\cH $ as in
  Theorem~\ref{main}.

(i) If $\Lambda \subseteq X $ is a set of stable sampling for $\cH $,
then
\begin{equation}
  \label{eq:rev1}
  \liminf _{r\to \infty } \inf _{x\in X} \frac{\# (\Lambda
      \cap B_r(x))}{\int _{ B_r(x)} k(y,y) \, d\mu (y)}  \geq 1 \, .
\end{equation}

(ii) If $\Lambda \subseteq X $ is a set of interpolation  for $\cH $,
then
\begin{equation}
  \label{eq:rev2}
  \limsup _{r\to \infty } \sup _{x\in X} \frac{\# (\Lambda
      \cap B_r(x))}{\int _{ B_r(x)} k(y,y) \, d\mu (y)}  \leq 1 \, .
\end{equation}
\end{cor}

\begin{proof}
 We only prove (i), as (ii) is similar. Dividing~\eqref{eq:e5} yields
\begin{align*}
   1 \leq
(1+ \epsilon C_1 )  \frac{\# (\Lambda \cap B_r)}{\int _{B_r} k(y,y) d\mu (y) } +
\epsilon C_2 \frac{\mu (B_r)}{\int _{B_r} k(y,y) d\mu (y)} + C_3   C_{\rho ,\Lambda } \,\frac{ \mu
(B_{r+R+\rho} \setminus B_{r-R-\rho}) }{\int _{B_r} k(y,y) d\mu (y) }
\, .
\end{align*}
Since $\int _{B_r} k(y,y) d\mu (y) \geq C_1 \mu(B_r)$ by
\eqref{eq_AD}, the second term on the right-hand side  is of order
$\epsilon $, and the third term tends to $0$ for $r\to \infty $.
Taking  the infimum over all $x\in X$ and letting $r$
tend to $\infty$, we obtain
$$
1\leq (1+C_1\epsilon )   \liminf _{r\to \infty } \inf _{x\in X} \frac{\# (\Lambda
      \cap B_r(x))}{\int _{ (B_r(x))} k(y,y) \, d\mu (y)}  + \epsilon
    C' \, ,
$$
which yields assertion (i).
\end{proof}
This corollary suggests that one could define the modified  Beurling
density of a set $\Lambda $ by \eqref{eq:rev1} and
\eqref{eq:rev2}. The corresponding density theorem is then
dimension-free with  critical density  $1$ independent
of the geometry of the \rkhs . By contrast, the critical density in
Theorem~\ref{main} depends on the reproducing kernel.

\begin{rem}
{
\normalfont
Instead of the Beurling densities one may
also apply an ultra-filter to \eqref{eq:8} and \eqref{eq:e5} and
obtain a density theorem with respect to a so-called frame measure function.
See~\cite{BCHL06,bala07} for the notion of frame measure function
and its applications to the comparison of frames.
}
\end{rem}

\begin{rem}
\label{rem_nohap}

Theorem \ref{main} (ii) is valid without axiom (HAP).

{\normalfont
Our proof of the density theorem emphasized the symmetry between
sampling and interpolation.
We have seen that the \emph{same} estimates are used in both density theorems. If we give up this symmetry, we can
streamline the proof of the interpolation part a bit, and
deduce the density conditions without assuming the kernel axiom (HAP).\\

Indeed, with the notation of the preceding proof let
$V_{B_r}= \operatorname{span}\set{k_\lambda \colon \lambda \in \Lambda \cap B_r }$
and $P$ the orthogonal projection onto $V_{B_r}$. The (unique) biorthogonal basis in $V_{B_r}$ is
$\set{P{g_\mu} \colon g_\mu \in \Lambda \cap B_r}$,
because
$\inprod{k_\lambda,P{g_\mu}}= \inprod{k_\lambda,{g_\mu}}=\delta_{\lambda,\mu}$
for
$\lambda, \mu \in \Lambda \cap B_r$. Choose $R = R(\epsilon )$ so that
that axiom (WL) is satisfied. Then
\[
\#(\Lambda \cap B_r)=\sum_{\lambda \in \Lambda \cap B_r}\inprod{k_\lambda,P{g_\lambda}}
 =\Big(\int_{B_{r+R}} +\int_{X \setminus B_{r+R}}\Big) \Big(\sum_{\lambda \in \Lambda \cap B_r}k_\lambda(y)
\overline{P{g_\lambda}(y)}\Big)
d\mu(y) = I + \tilde L \,.
\]
By \eqref{eq:a0} and \eqref{eq:a1} we obtain
$\abs{\tilde L} \leq \epsilon C_1 \#(\Lambda \cap B_r)$,
whereas \eqref{eq:c45} yields
\[
I = \int_{B_{r+R}} k(y,y) d\mu(y) \leq   \int_{B_{r}} k(y,y) d\mu(y) +
\sup _{x\in X} k(y,y) \mu \big(B_{r+R}\setminus B_r\big) \, .
\]
Consequently,
\[
\#(\Lambda \cap B_r) \leq \int_{B_r} k(y,y) d\mu(y)+ C_0 \mu \big(B_{r+R}\setminus B_r\big) + \epsilon C_1 \#(\Lambda \cap B_r)\,,
\]
which readily yields
\[
D^{\pm}(\Lambda) \leq \operatorname{tr}^{\pm}(k) \,.
\]
}
\end{rem}

\subsection{Off-diagonal decay with respect to a metric}
In applications, the reproducing kernel often possesses some off-diagonal decay. In this case the kernel
axioms are easier to check. The following proposition shows that Theorem \ref{tmintro} is a special case of Theorem
\ref{main}.

\begin{proof}[Proof of Theorem \ref{tmintro}]
We show that the hypothesis of Section \ref{sec_assumptions} are satisfied.
The assumption on $d$ and $k$ clearly implies the weak localization condition (WL)
and the diagonal condition (D). It only remains to check the
homogeneous approximation property (HAP). Assume that $\Lambda \subseteq X$ is such that
$\{k_\lambda: \lambda \in \Lambda\}$ is a Bessel sequence.

As noted in Remark \ref{rem_lemma_nohap}, the proof of Lemma \ref{lemma_sam_rel} does
not depend on Axiom (HAP). Hence we can invoke that lemma to obtain $\rho>0$ such that \eqref{eq:def_rel_sep} holds. Similarly, we can invoke Lemma \ref{lemma_loc_do} - which, as noted in Remark \ref{rem_ld}, only depends on Axiom (WAD) - to further grant that
\begin{align}
\label{eq_b}
\mu(B_{2\rho}(x)) \leq C_\rho \mu(B_\rho(x)), \qquad x \in X.
\end{align}

We observe first that the obvious inequality
\[
1+ d(x,y) \leq (1+ d(x,\lambda))(1+ d(\lambda,y)) \quad \text{for all}\quad x,y, \lambda \in X
\]
implies that
\[
(1+ d(x,\lambda))^{-2 \sigma} \leq (1+ d(\lambda,y))^{2 \sigma} (1+ d(x,y))^{-2 \sigma}\,.
\]
Therefore,
\begin{align*}
  \abs{k(x,\lambda)}^2
\leq C (1+ d(x,\lambda))^{- 2 \sigma}
&= \frac {C} {\mu (B_\rho(\lambda))}\int_{B_\rho(\lambda)} (1+ d(x,\lambda))^{- 2 \sigma} d\mu(y)\\
&\leq \frac {C} {\mu (B_\rho(\lambda))}\int_{B_\rho(\lambda)}
                \frac{(1+ d(\lambda,y))^{ 2 \sigma}}{(1+ d(x,y))^{ 2 \sigma}} d\mu(y) \\
&\leq C\frac {(1+\rho)^{2 \sigma}} {\mu (B_\rho(\lambda))}\int_{B_\rho(\lambda)}
                {(1+ d(x,y))^{ -2 \sigma}} d\mu(y) \,.
\end{align*}
Consequently,
\begin{align*}
&\sum_{\lambda \in \Lambda \cap (X\setminus B_r(x))} \abs{k(x,\lambda)}^2 \\
&\qquad \leq C(1+ \rho)^{2 \sigma} \int_X
\Big(\sum_{\lambda \in \Lambda \setminus B_r(x)}{\mu (B_\rho(\lambda))}^{-1}
\ind_{B_\rho(\lambda)}(y) \Big)
(1+d(x,y))^{- 2 \sigma} d\mu (y) \,.
\end{align*}
We note that the sum vanishes if $d(x,y) \leq r- \rho$, thus the integral can be taken  over the set $X \setminus
B_{r-\rho}(x)$. Next we estimate the sum for fixed $y \in X \setminus B_{r-\rho}(x)$. Note that
if $y \in B_\rho(\lambda)$, then $B_\rho(y) \subseteq B_{2 \rho} (\lambda)$, and,
by \eqref{eq_b},
\begin{align*}
\mu(B_\rho(\lambda)) \geq C_\rho^{-1}
\mu(B_{2\rho}(\lambda)) \geq C_\rho^{-1} \mu(B_\rho(y)).
\end{align*}
Hence, using \eqref{eq:def_rel_sep}, we can estimate
\begin{align*}
\sum _{\lambda \in \Lambda \setminus B_r(x)}\frac {1} {\mu (B_\rho(\lambda))}
\ind_{B_\rho(\lambda)}(y)
&\leq \frac{C_\rho}{\mu(B_\rho(y))} \sum_{\lambda \in \Lambda} \ind_{B_\rho(y)}(\lambda)
\leq C C_\rho.
\end{align*}
In conclusion,
\begin{align*}
   \sum_{\lambda \in \Lambda \setminus  B_r(x)} \abs{k(x,\lambda)}^2
\leq C C_\rho (1+ \rho)^{2 \sigma} \int_{X \setminus B_{r-\rho}(x)}
(1+d(x,y))^{- 2 \sigma} d\mu (y) \,.
\end{align*}
By hypothesis, this expression tends to zero uniformly in $x$ as $r \to \infty$, whence $k$ satisfies (HAP).
\end{proof}

\section{Examples}
\label{sec_apps}

In this section we discuss several examples of density theorems from
different  areas of analysis. Our point is to show that some of the
fundamental density theorems in signal analysis, complex analysis,
frame theory, and harmonic analysis follow from the axiomatic
approach. All we have to do is to check the general conditions of
Section~\ref{sec_assumptions} and formulate the corresponding theorem. This is not
always easy, and our discussion will point out some of the
difficulties and pitfalls.

\subsection{Bandlimited Functions}

 Let $\Omega\subseteq \rd $ be  measurable with finite
Lebesgue measure and $B_\Omega = \{ f\in \lrd : \supp \hat{f}
\subseteq \Omega \}$ be the corresponding Paley-Wiener space.  As
observed in the introduction, its
reproducing kernel is
$$
k(x,y) = \int _\Omega e^{-2\pi i \xi \cdot (y - x)} \, d\xi \, .
$$
Clearly $X= \rd $ with the Euclidean distance and Lebesgue measure
$d\mu (x)  = dx$  satisfy  the geometrical  assumptions. As for the
kernel, we have
$k(x,x) = |\Omega |$ and thus the averaged trace is obviously  $\mathrm{tr}^+ (k) = \mathrm{tr}^- (k) =
\mu (\Omega)$.

The verification of the weak localization of
the kernel is easy, because $k(x,y) = \widehat{1 _\Omega }(y-x)$ where
$\widehat{1 _\Omega } $ is the Fourier transform of an $L^2$-function. Therefore
$$ \int _{\rd \setminus B_r(x)} |\widehat{\ind_\Omega } (y-x)|^2 \, dy = \int _{|y|
  \geq r} |\widehat{\ind_\Omega } (y)|^2 \, dy < \epsilon ^2
$$
for suitably large $r$. The axiom (HAP) is more subtle. In fact, it
holds for bounded spectrum $\Omega $, where it is a consequence of
the Plancherel-Polya inequality for entire functions of exponential
growth~\cite{triebel83}. However, one can show that (HAP)  fails for
unbounded spectra, therefore Theorem~\ref{main} is not directly
applicable. In this case one applies the sampling part of the density theorem to the
subspace $B_{\Omega \cap B_R(0)} \subseteq B_\Omega$ and then takes
the limit $R\to \infty $. For the interpolation part, we do not need
axiom (HAP) - cf. Remark \ref{rem_nohap} - so there is no difficulty for unbounded spectrum $\Omega
$, see~\cite{NO12} for the details.

To summarize, Theorem~\ref{main} implies Landau's fundamental density
theorem for bandlimited functions. The geometric properties and all
but one property of the kernel are obvious, but the homogeneous
approximation property requires some mathematical arguments.

\subsection{Functions of Variable Bandwidth}

Next we consider the spectral subspaces of the Schr\"odinger operator
$D_q f = -\frac{1}{4 \pi^2} f'' + q f$ in dimension $1$ with a compactly supported
potential $q \in C^2$. Let $\Omega \subseteq \bR ^+$ be a bounded set
and let $PW_\Omega (D_q)$ be the spectral subspace
corresponding to  spectrum $\Omega $. If $q \equiv
0$, then $D_0= - \frac{1}{4 \pi^2}\frac{d^2}{dx^2}$ is diagonalized by the Fourier
transform $\cF$ so that $\cF   D_0 \cF ^{-1} f (\xi )=  \xi ^2
\hat{f}(\xi ) $ is  the  operator of multiplication by
$ \xi ^2$. For the spectral subspace $PW_\Omega (D_0) $ only the
spectral values $\xi ^2 \in \Omega $ are relevant, therefore
$PW_\Omega (D_0) = \{ f \in L^2(\bR) : \mathrm{supp} \widehat{f}
\subseteq \Omega ^{1/2} \} = B_{\Omega ^{1/2}}$ is  the  Paley-Wiener space of
bandlimited functions with spectrum in $\Omega ^{1/2} = \{ \xi \in \bR
: \xi ^2 \in \Omega \}$.  One can show that  $PW_\Omega (D_q)$ is a
\rkhs .

If $q \not \equiv 0$, then we may consider $PW_\Omega (D_q)$ as a
perturbation of the Paley-Wiener space. It is therefore natural to
expect that the same density conditions for sampling and interpolation
also hold  for $PW_\Omega (D_q)$.  Indeed, in~\cite{GK15} we
proved the following result.

  \begin{tm}\label{tm-vbw}
  Assume that $\Omega \subseteq \bR ^+$ is a bounded set with
  positive (Lebesgue) measure.
\begin{itemize}[itemindent=0cm, leftmargin=1cm, itemsep=0.2cm, parsep=0cm]
\item[(i)]    If $\Lambda $ is a set of sampling for  $PW_\Omega (D_q)$, then
$D^-(\Lambda ) \geq \mu (\Omega ^{1/2})$.

\item[(ii)]   If $\Lambda $ is a set of interpolation for  $PW_\Omega (D_q)$, then
$D^+(\Lambda ) \leq \mu (\Omega ^{1/2}) $.
\end{itemize}
\end{tm}

Whereas this result is expected, its proof is surprisingly difficult.
In contrast to the Paley-Wiener space $B_\Omega = PW_\Omega (D_0)$,   the
reproducing kernel for  $PW_\Omega (D_q)$ is not known explicitly. To
derive  Theorem~\ref{tm-vbw},  we had to
use the fine details of the  scattering theory for one-dimensional Schr\"odinger
operators  for the  verification of  the kernel axioms (WL) and
(HAP) and for the
computation of  the averaged trace of the kernel.

 Thus for this example our
main efforts in \cite{GK15} were devoted to deriving suitable kernel
estimates.

\begin{rem} {\rm
(i) For  $PW_\Omega (D_q)$ one can also derive  sufficient conditions for
sampling.  See~\cite{Pes01} for a qualitative sampling theorem and
\cite{GK15} for an explicit almost optimal sampling theorem.

(ii)  In \cite{GK15} we treated a unitarily equivalent
model of the Paley-Wiener space and studied a new concept of variable
bandwidth. In that case the density results are formulated differently
and also involve a different  geometry.   }
\end{rem}

\subsection{Sampling in locally compact groups}

Let $\cG $ be a  locally compact group with Haar
measure $d\mu = dx $. We make the following additional assumptions:

\begin{itemize}[itemindent=0cm, leftmargin=1cm, itemsep=0.2cm, parsep=0.15cm]
\item[(i)]  $\cG $ is compactly generated, i.e.,
there exists a symmetric neighborhood $U = U\inv $ of $e$
with compact closure such that
$\cG = \bigcup _{n=0} ^\infty U^n$.  The corresponding metric  on $\Gc$, the
so-called  word metric,  is defined as
\begin{align*}
d(x,y) := \min\{n \in \mathbb{N}_0: x^{-1} y \in U^n\}, \qquad x,y \in \Gc.
\end{align*}
It is clearly left-invariant, and the balls are compact sets, in particular Borel sets of finite measure.

\item[(ii)] $\cG $ has polynomial growth, i.e., there exist constants $C,D>0$
such that
\begin{equation}
  \label{eq:c88}
\mu(U^n) \leq C n^D, \qquad n \in \mathbb{N}.
\end{equation}
\end{itemize}
Under these assumptions the word metric possesses the weak annular
decay property. In fact, Tessera~\cite[Cor.~10]{tes07} showed that polynomial
growth implies the annular decay property. Thus $(\cG, d , \mu )
$ satisfies all geometric axioms. (See also \cite{comi98, br14}.)

Now let $\pi $ be an irreducible,  unitary, \emph{square-integrable}
representation on a Hilbert space $\cH $. The orthogonality relations
for square-integrable representations~\cite{Fu_LN} allow us to identify $\cH $
with a reproducing kernel Hilbert space. Precisely, fix a non-zero
$g\in \cH $ with normalization $\|g\| = 1$  and consider the map
$$
\mC: \cH \to L^2(\cG ), \quad \mC f(x) =
\langle f, \pi (x) g \rangle , \qquad x\in \cG \, .
$$
The orthogonality relations then imply that
\begin{equation}
  \label{eq:c76}
\langle \mC f, \mC h \rangle _{L^2(\cG )} = d_\pi ^{-1} \,   \langle f, h \rangle
_\cH \, ,
\end{equation}
where the  constant $d_\pi $ is the so-called formal dimension of $\pi $.
 Consequently $\mC$ is a multiple of an isometry, and we can
identify the representation space $\cH $ with the subspace $\tilde \cH
= \mC \cH $ of $L^2(\cG )$. Now
choosing $h = \pi (x) g$ for $x\in \cG $ in~\eqref{eq:c76}, we obtain that
\begin{align*}
  \mC f(x) = \langle f, \pi (x) g\rangle_{\cH} = d_\pi \langle \mC f, \mC (\pi
  (x)g)\rangle_{L^2(\cG )} = d_\pi \int _\cG \mC f(y) \overline{\langle \pi (x)g, \pi (y)g\rangle} \, dy \, .
\end{align*}
This identity says that $\tilde \cH $ is a reproducing kernel Hilbert
space with kernel
$$
k(x,y) =  d_\pi \, \overline{\langle \pi (x)g, \pi (y)g\rangle}  =
d_\pi \langle g,
\pi (y\inv x ) g\rangle \, .
$$
Consequently, $k(x,x) = d_\pi $ is constant,  and axiom $(D)$ is  satisfied
trivially. The computation of the averaged trace is a banality and yields
$$
\mathrm{tr}^{\pm} (k) = \frac{1}{\mu (B_r(x) )} \int _{B_r(x)} k(y,y)
\, d\mu (y)   = d_\pi \, .
$$

Moreover, since $x \to \langle g, \pi (x)g\rangle $ is in
$L^2(\cG )$,  the  weak localization (WL) is also  satisfied. Again,
the  homogeneous approximation property (HAP) is the least obvious property
and requires some work. Let $\mathbf{B}$ consist of all vectors $g \in
\cH $ of the form $g = \int _\cG \eta(x) \pi (x) g_0 \, d\mu (x)$ for
some $g_0 \in \cH $ and $\eta$ a compactly supported continuous function on
$\cG $. If $g\in \mathbf{B}$ and $\Lambda \subseteq \cG $ is an
arbitrary relatively separated set, then the set of reproducing kernels
$\{ \langle g,  \pi (\lambda ^{-1} \cdot)g\rangle : \lambda
\in \Lambda \}$ satisfies axiom (HAP) by an observation in~\cite{Gr_HAP}.

To formulate  Theorem~\ref{main} for this particular
example, we finally note that $\Lambda \subseteq \cG $ is a set of
sampling (set of interpolation) for $\tilde \cH $ \fif\ $\{ \pi
(\lambda )g:
\lambda \in \Lambda \}$ is a frame (Riesz sequence)  for $\cH $.
Frames of this form are often called coherent frames or discrete
subsets of coherent states.
 Theorem~\ref{main} yields the following density result for
coherent frames.

\begin{tm} \label{tm-coherent}
Let $\cG $ be a  compactly generated,  locally compact group with
polynomial growth, and let $\pi $ be an irreducible,  unitary,
\emph{square-integrable} representation on a Hilbert space $\cH $.

(i)  If $\{\pi(\lambda)g:\lambda\in\Lambda\}$ is a frame for $\cH$ for
$g\in \mathbf{B}$, then
$D^{-}(\Lambda) \geq d_\pi $.

(ii) If $\{\pi(\lambda)g:\lambda\in\Lambda\}$ is a Riesz sequence in
$\cH$, then
$D^{+}(\Lambda) \leq d_\pi $.
\end{tm}

This result seems to be new. For square-integrable representations of
groups of polynomial growth it  provides a critical density that
separates frames from Riesz sequences.   For concrete representations, e.g., the
Schr\"odinger representation of the Heisenberg group
Theorem~\ref{tm-coherent}  has been
derived many times in the context of Gabor analysis~\cite{heil07}. For
homogeneous (nilpotent) groups it has been proved in the thesis of A.\
H\"ofler~\cite{hoefler} by using the techniques of Ramanathan and
Steger~\cite{RS95}.

Let us mention that
 the construction of coherent frames associated to irreducible
 representations was first studied  systematically in  coorbit theory, see
 \cite{gro91,FeiGr_coorbit2}. If $\Lambda \subseteq \cG $ is
 ``sufficiently dense'', then $\{\pi(\lambda)g:\lambda\in\Lambda\}$ is
 a frame for $\cH $. Theorem~\ref{tm-coherent} complements the
 existence of such frames by a critical density.

Theorem~\ref{main} also  yields several new density  results about sampling and
interpolation in reproducing kernel Hilbert spaces that are
invariant under a group action. As the full exploitation of
Theorem~\ref{main} is beyond the scope of this  section, we will come
back to it in further work.

\subsection{Complex analysis}
Finally, we deal with sampling and interpolation in weighted spaces
of analytic functions. We will partially  rederive  Lindholm's result~\cite{Lin01}  from Theorem~\ref{main}.

Let $\phi$ be a plurisubharmonic function on $\mathbb{C}^n$ which is $2$-homogeneous and $C^2$ on $\mathbb{C}^n
\setminus \{0 \}$.  We also assume that there exist $A, B > 0$ such that
\begin{equation} \label{subharm}
A \cdot \mathrm{Id}_n \leq \Big( \partial_j \bar \partial_k \phi(z)
\Big)_{j,k = 1, \dots n} \leq B \cdot \mathrm{Id}_n
\end{equation}
for all  $z \neq 0$, in the sense of positive definite matrices.
It follows  that
$A^n \leq \det (\partial_j \bar \partial_k \phi)_{jk} \leq B^n$
on $\mathbb{C}^n \setminus \{ 0 \}. $ Note that in dimension $n=1$
this condition simply means that the Laplacian $\Delta \phi = \partial
\bar \partial \phi $ is
bounded above and below from $0$.

Our main object is the  Hilbert space $\mathcal{F}^2_\phi$
of entire  functions on $\bC ^n$ defined by the norm
$ \|f \|_{\mathcal{F}^2_\phi } ^2 = \int _{\bC ^n} |f|^2 e^{-2\phi } dm $. The
standard example is the weight
$\phi (z) = |z|^2/2$ which yields the Bargmann-Fock space. This is a
reproducing kernel Hilbert space with kernel $\frac{1}{\pi}e^{z\cdot \bar
  {w}}$. Since this kernel is unbounded, we use a different
normalization to put it into the framework of Theorem~\ref{main}.

We take $X= \mathbb{C}^n$ with the usual Euclidean distance and  the
measure  $d\mu = \det (\partial_j
\bar \partial_k \phi)_{jk} dm$ where $dm$ is the Lebesgue measure.  Thus   $\mu $ is equivalent
to Lebesgue measure. Let $A^2_\phi $ consist of all functions
 of the form
\begin{equation} \label{weightedfunctions}
g=  \frac{1}{\sqrt{\det (\partial_j \bar \partial_k \phi)_{jk}}}f e^{-\phi},
\end{equation}
 where $f$ is entire, such that
\begin{equation} \label{weightedanalnorm}
\|g\|^2  = \int_{\mathbb{C}^n} |g|^2 d \mu =
 \int_{\mathbb{C}^n} |f|^2 e^{-2\phi} dm < \infty.
 \end{equation}
We observe immediately that the assumptions on the metric and the
measure required in the main theorem are satisfied.

It can be shown that  $A^2_\phi$ is a reproducing kernel Hilbert space.
We denote the kernel by $K= K_\phi$.  For the weight $\phi (z) =
|z|^2/2$, our normalization  yields the following explicit expression
for the kernel \[
K(z,w)= \frac{2^n}{\pi^n} e^{z \cdot \bar w-|z|^2/2-|w|^2/2}.
\]
It is easy to see that this kernel satisfies the axioms (D),  (WL),  and
(HAP), therefore the Seip's  necessary density
conditions~\cite{seip92} for sampling in Bargmann-Fock space follow
 (without strict inequalities) directly from
Theorem~\ref{main}.

For more general weights $\phi$ there is no  explicit formula for the
kernel, but strong estimates are known. We use    Lindholm's
estimates~\cite{Lin01}.  Since he  works with the measure $e^{-\phi } dm$ and
entire functions, we have to translate these results to our
normalization with the measure $\mu $ and functions of the form
\eqref{weightedfunctions}. Using the observation~\eqref{eq:v3}, the relation between the  kernel $B_\phi $
in~\cite{Lin01} and our kernel $K_\phi $ is given by
\[
K_{\phi}(z,w) = \frac{1}{\sqrt{\det (\partial_j \bar \partial_k \phi)_{jk}(z)\det (\partial_j \bar \partial_k \phi)_{jk}(w)}} B_{\phi}(z,w) e^{-\phi(z)-\phi(w)}.
\]
Translated to our notation, Lindholm~
\cite{Lin01} proved the following facts about $K_\phi
$:

(i)   There exist  constants $C, T>0$  depending on $A$ and $B$ in
\eqref{subharm}, such that for all $k>0$ holds the
 decay  estimate
\begin{equation} \label{offdiag}
K_{k^2 \phi}(z, w) \leq C e^{-k T |z-w|} \, .
\end{equation}

(ii) On the diagonal the kernel satisfies the limit relation
\begin{equation} \label{diagconv}
\lim _{k\to \infty } K_{k^2 \phi} (z, z) =  \frac{2^n}{\pi^n} \, ,
 \end{equation}
with uniform convergence on   $\mathbb{C}^n \setminus B_\tau(0)$ for
arbitrary $\tau >0$.

In addition,
 the $2$-homogeneity and a simple change of variables imply that
 \begin{equation} \label{changevar}
K_{k^2 \phi}(z, w) = K_{\phi}( kz, kw) \, .
\end{equation}

The axioms (WL) and (HAP) follow immediately from the off-diagonal
decay \eqref{offdiag} of the kernel $K_\phi $, likewise $K_\phi (z,z)$
is bounded.
The lower bound in Axiom (D) \eqref{eq_AD} can now be verified in the following way.
Note first that since $g= \left(\det (\partial_j \bar \partial_k \phi)_{jk}\right)^{-1/2} e^{-\phi}
\in A^2_\phi$, it follows that $K(z,z) \not= 0$, for all $z \in \mathbb{C}^n$. In addition,
by \eqref{diagconv} and
\eqref{changevar}, there exist $c, \tau, R > 0$ such that
$K_{R^2\phi } (z,z) = K(Rz, Rz) > c$ on $\mathbb{C}^n \setminus B_\tau(0)$. Then $K_\phi(z,z)=
K_\phi (R \cdot z/R, R \cdot z/R) > c$ for any $|z| > R \tau$.
By continuity
and the fact that $K(z,z) \neq 0$ for all $z$, we also have $K(z,z) > c_2$ for some $c_2>0$ for $|z|
\leq R \tau$.

Thus the geometry and the kernel satisfy all required hypotheses of
Section~\ref{sec_assumptions}. Thus Theorem~\ref{main} is
applicable and yields a critical density that separates sampling from
interpolation. It remains to compute this critical density.

\begin{lemma} \label{trace-plu}
  If the weight $\phi $ is plurisubharmonic, $2$-homogeneous, and
  satisfies \eqref{subharm}, then
  \begin{equation}
    \label{eq:d6}
    \mathrm{tr}^{+}(K_\phi )  =     \mathrm{tr}^{-} (K_\phi ) = \frac{2^n}{\pi^n} \, .
  \end{equation}
\end{lemma}
\begin{proof}
We use the homogeneity \eqref{changevar} and
$(\partial_j \bar \partial_k \phi)(rz) = \partial_j
\bar \partial_k \phi(z)$ for all $r >0$. Then
\begin{align} \label{weightedanaltrace}
\sup_{x \in \mathbb{C}^n}  & \frac{1}{ \int_{B_r(x)}\det (\partial_j
  \bar \partial_k \phi)_{jk} dm} \int_{B_r(x)} K_\phi(z,z) d \mu (z)
\\
&=  \sup_{x \in \mathbb{C}^n}
\frac{1}{ \int_{B_1(x/r)}\det (\partial_j \bar \partial_k \phi)_{jk}
  dm}\int_{B_1(x/r)} K_{r^2\phi} (w, w) d \mu(w) \notag \\
& = \sup_{y \in \mathbb{C}^n}  \frac{1}{ \int_{B_1(y)}\det (\partial_j
  \bar \partial_k \phi)_{jk} dm}\int_{B_1(y)} K_{r^2\phi} (w, w)
d\mu(w). \notag
\end{align}
Now,  we use the fact that $K_{r^2 \phi}(w,w)$ converges to $\frac{2^n}{\pi^n}$
uniformly outside any ball $B_\tau(0)$, $\tau >0$. If $B_1(y)$
contains the origin, we remove a small neighborhood of $0$, otherwise
we use \eqref{diagconv} directly.   Given $\epsilon >0$, it follows that
\[
\bigg| \sup_{x \in \mathbb{C}^n} \frac{1}{ \int_{B_r(x)}\det
  (\partial_j \bar \partial_k \phi)_{jk} dm} \int_{B_r(x)} K_{r^2
  \phi}(z,z) d \mu(z) -  \frac{2^n}{\pi^n}   \bigg| \\
\leq \epsilon + \mathrm{o}(1)
\]
as $r \to \infty$. Therefore,
\[
\limsup_{r \to \infty} \bigg| \sup_{x \in \mathbb{C}^n} \frac{1}{
  \int_{B_r(x)}\det (\partial_j \bar \partial_k \phi)_{jk} dm}
\int_{B_r(x)} K_{\phi}(z,z) dm(z) -\frac{2^n}{\pi^n}   \bigg|  \leq
\epsilon \\
\]
for all $\epsilon >0$, which means that
\[
\mathrm{tr}^+(K_\phi ) = \lim_{r \to \infty}  \sup_{x \in \mathbb{C}^n} \frac{1}{ \int_{B_r(x)}\det (\partial_j \bar
\partial_k \phi)_{jk} dm} \int_{B_r(x)} K_\phi(z,z) dm(z)  =   \frac{2^n}{\pi^n}.
\]
Likewise $\mathrm{tr}^- (K_\phi ) = \frac{2^n}{\pi^n} $.
  \end{proof}

Theorem~\ref{main} now implies Lindholm's result~\cite{Lin01}.

\begin{tm}
Assume that $\phi $ is plurisubharmonic, $2$-homogeneous, and satisfies
\eqref{subharm}.

\begin{itemize}[itemindent=0cm, leftmargin=1cm, itemsep=0.2cm, parsep=0cm]
\item[(i)] If $\Lambda \subseteq \bC ^d$ is a set of sampling for $A_\phi
^2$, then $D^- (\Lambda ) \geq \frac{2^n}{\pi^n}$.

\item[(ii)] If $\Lambda \subseteq \bC ^d$ is a set of interpolation for $A_\phi
^2$, then $D^+ (\Lambda ) \leq \frac{2^n}{\pi^n}$.
\end{itemize}
\end{tm}

\begin{rem}
{\normalfont
For generalized Fock spaces in one complex variable, Ortega-Cerd\`{a} and Seip
\cite{ortega1998beurling} and Marco, Massaneda, Ortega-Cerd\`{a}
\cite{marco2003interpolating}  proved a density  theorem for
non-homogeneous weights as well. Although Theorem~\ref{main} applies,
we are (not yet) able to recover their explicit result. This would
require to derive a version of Lemma~\ref{trace-plu} for non-homogeneous
doubling weights.
}
\end{rem}

\subsection{Density of Abstract Frames}
Finally, we note certain connections with the density theory for abstract frames \cite{bacahela06,BCHL06,bala07}.

Let $\cH $ be a separable Hilbert space, $(X,d)$ a countable
metric space with counting measure $\mu $, and
$\mathcal{F}=\{ f_x : x\in X\}$ a frame for $\cH$, i.e.,
there exist $A,B>0$ such that
\begin{equation}
  \label{eq:d2}
A \|f\|^2 \leq \sum _{x\in X} |\langle f, f_x \rangle |^2 \leq B
\|f\|^2 \qquad \forall f \in \cH \, .
\end{equation}
Using the coefficient operator $\mC: \cH \to \ell ^2(X)$
\begin{equation}
  \label{eq:d1}
  \mC f(x) = \langle f, f_x\rangle   \qquad x\in X \, ,
\end{equation}
we can identify the abstract Hilbert space $\cH $ with the subspace of
functions $\mC f $ of $\ell ^2(X)$. By the frame inequalities
\eqref{eq:d2} $\mC$ is one-to-one with closed range in $\ell ^2(X)$,
which we call $\widetilde \cH  = \mC \cH \subseteq \ell ^2(X)$.

 Let
$\{\tilde f_x : x\in X\}$ be the (canonical) dual frame of $\cH $,
then every $f\in \cH $ possesses the frame expansion
$ f = \sum _{y\in X} \langle f, f_y \rangle \tilde f_y$, and
consequently
$$
\mC f(x)  = \sum _{y\in X} \langle f, f_y \rangle \langle \tilde f_y ,
f_x\rangle  = \sum _{y\in X} \mC f( y)  \langle \tilde f_y ,
f_x\rangle  \, ,
$$
This means that $\tilde \cH $ is a reproducing kernel subspace of $\ell^2(X)$ with kernel
\begin{equation*}
    k(x,y) = \langle \tilde f_y , f_x \rangle.
\end{equation*}
The two properties (WL) and (HAP) for the pair $(X,\widetilde \cH)$ are equivalent to what
in \cite{bacahela06} is called $\ell^2$-\emph{localization} of the frames
$\mathcal{F}$ and $\mathcal{\widetilde F}$. Furthermore, the (lower) averaged
trace of this kernel is
$$
\mathrm{tr}^- (k) = \liminf _{r\to \infty } \inf _{x\in X} \frac{1}{\# B_r(x)} \sum _{y\in B_r(x)} \langle
    \tilde f_y, f_y\rangle \, .
$$
This quantity correspond exactly to the (lower)
\emph{frame measure} of $\mathcal{F}$ in
\cite{bacahela06,bala07}.

Besides these technical similarities, Theorem \ref{main} is not formally comparable to the
results in \cite{bacahela06,BCHL06,bala07}.  The theory of Balan,
Casazza, Heil, and Landau  in \cite{bacahela06,BCHL06,bala07} compares two abstract
frames, and derives an equality relating density and measure. By contrast,
Theorem \ref{main} compares a frame of reproducing kernels to a possibly continuous resolution of the
identity.

\subsection{More on Axiom (WAD) ---  The
standard Bergman space on the upper-half plane}
\label{exa}
Let $X=\{z \in \mathbb{C}: \mathrm{Im}(z)>0\}$ with the hyperbolic
distance
\begin{align*}
d(z,w) = 2 \tanh^{-1} \left( \frac{\abs{z-w}}{\abs{z - \overline{w}}}\right),
\end{align*}
and measure $d\mu(z) = \frac{1}{\pi} \mathrm{Im}(z)^{-2} dA(z)$, where
$dA(z)$ denotes the Lebesgue measure. We consider the
RKHS of functions
\begin{align*}
\cH = \left\{ \mathrm{Im}(z) f(z), \mbox{ with }f: X \to \mathbb{C}
\mbox{ analytic } \right\} \cap L^2(X, \mu).
\end{align*}
One can readily verify that the measure
of $B_R(0)$ grows exponentially in $R$ and that the weak annular decay property does not hold. Hence, Theorem \ref{main} is not applicable in this setting. Nevertheless, with the appropriate notion of density introduced by Seip \cite{seip93},
necessary and sufficient conditions for sampling and interpolation do hold for $\cH$.

\end{document}